%% file: AblationStudy.tex
\documentclass{article}

\usepackage[utf8]{inputenc} %
\usepackage[T1]{fontenc}    %
\usepackage{hyperref}       %
\usepackage{url}            %
\usepackage{booktabs}       %
\usepackage{amsfonts}       %
\usepackage{nicefrac}       %
\usepackage{microtype}      %
\usepackage{xcolor}         %
\usepackage{fullpage}
\usepackage{amsmath,amssymb,amsthm}
\usepackage{algorithm,algpseudocode}
\usepackage{graphicx}
\usepackage{float}

\include{MacroBook}

\usepackage{multirow}
\title{A Structured Tour of Optimization with Finite Differences}

\author{Marco Rando\thanks{Malga - DIBRIS, University of Genova, IT
		({\tt marco.rando@edu.unige.it}, {\tt lorenzo.rosasco@unige.it}).}
	\and Cesare Molinari\thanks{MaLGa - DIMA, University of Genova, Italy 
		({\tt cesare.molinari@edu.unige.it}, {\tt silvia.villa@unige.it}).}
	\and Lorenzo Rosasco\footnotemark[1] \thanks{Istituto Italiano di Tecnologia, Genova, Italy and CBMM - MIT, Cambridge, MA, USA}
    \and Silvia Villa\footnotemark[2]
}

\date{}

\begin{document}

\maketitle

\begin{abstract}
\noindent Finite-difference methods are widely used for zeroth-order optimization in settings where gradient information is unavailable or expensive to compute. These procedures mimic first-order strategies by approximating gradients through function evaluations along a set of random directions. From a theoretical perspective, recent studies indicate that imposing structure (such as orthogonality) on the chosen directions allows for the derivation of convergence rates comparable to those achieved with unstructured random directions (i.e., directions sampled independently from a distribution). Empirically, although structured directions are expected to enhance performance, they often introduce additional computational costs, which can limit their applicability in high-dimensional settings. In this work, we examine the impact of structured direction selection in finite-difference methods. We review and extend several strategies for constructing structured direction matrices and compare them with unstructured approaches in terms of computational cost, gradient approximation quality, and convergence behavior. Our evaluation spans both synthetic tasks and real-world applications such as adversarial perturbation. The results demonstrate that structured directions can be generated with computational costs comparable to unstructured ones while significantly improving gradient estimation accuracy and optimization performance.
\end{abstract}

\section{Introduction}

Zeroth-order optimization problems are a class of problems in which only evaluations of the objective function are accessible, while gradient information is either unavailable or prohibitively costly to obtain. Such settings commonly occur when the function is evaluated via simulation or when the analytical form of the gradient is unavailable or too expensive to compute \cite{intro_schein,gain_tuning_zth,salimans2017evolution,z0_sig_proc,NEURIPS2018_7634ea65}.

\noindent Various classes of algorithms, also known as zeroth-order methods \cite{nesterov2017random,adabkb,fornasier,Totzeck2022,intro_schein,ga_review,chen2015randomized}, have been developed to address these problems, with finite-difference methods being one of them. %
These iterative procedures mimic first-order optimization strategies by approximating gradients through finite differences computed along a set of random directions. Two main variants of finite-difference methods are commonly distinguished: unstructured and structured approaches. In unstructured methods, directions are sampled independently from a fixed probability distribution, such as a standard Gaussian or a uniform distribution over the unit sphere \cite{nesterov2017random,Berahas2022,shamir2017optimal,salimans2017evolution}. In contrast, structured methods impose additional constraints on the directions such as orthogonality. %
Although structured methods have been shown to offer better performance in certain settings \cite{Berahas2022,rando2023optimal,choromanski2019unifying}, they often come with higher computational costs, especially in high-dimensional problems. This presents a practical limitation for structured methods, since many modern applications of finite-difference methods - such as reinforcement learning \cite{salimans2017evolution,choromanski2019unifying}, adversarial machine learning \cite{bszd,rando2024newformulationzerothorderoptimization}, and fine-tuning of large language models \cite{mezo,zo_llm_benchmark} - are usually high-dimensional optimization problems. In such regimes, unstructured methods are widely used due to the simplicity and low cost of generating random directions. However, these procedures often come at the expense of high variance and less informative gradient estimates, which may result in slower convergence in practice. 

\noindent Recent works \cite{choromanski2019unifying,kozak2021zeroth,rando2022stochastic,stief_zeroth,stief_zth,rando2023optimal,rando2024newformulationzerothorderoptimization} have proposed structured finite-difference methods with efficient procedures for constructing direction matrices. Despite this progress, to the best of our knowledge, no comprehensive comparison of structured and unstructured strategies has been conducted, especially in the practically relevant regime where the number of directions $\ell$ is smaller than the input dimension $d$. As a result, it remains unclear which structured strategies are most effective and under what conditions they offer a tangible advantage.

\noindent In this work, we aim to fill this gap. We review and extend several procedures for constructing structured direction matrices and evaluate them empirically on both synthetic benchmarks and real-world tasks. Our comparison focuses on three key aspects: computational cost, gradient approximation quality, and convergence under a fixed budget of function evaluations. Our findings show that it is possible to design structured methods with computational costs comparable to unstructured ones while achieving significantly better performance in terms of gradient accuracy and optimization progress. This suggests that structured methods are a viable and often preferable alternative (even in high-dimensional settings) and opens up promising directions for their application to large-scale problems, such as fine-tuning large language models.

\noindent The paper is organized as follows. In Section \ref{sec:prob_setup}, we define the zeroth-order optimization problem and introduce finite-difference algorithm. We also review and extend methods for generating structured direction matrices. In Section \ref{sec:perf_metrics}, we describe the experimental setup and present the performance metrics used in our evaluation. In Section \ref{sec:results}, we provide the experimental results, and in Section~\ref{sec:conclusion}, we conclude with some final remarks.

\section{Finite-Difference Algorithms for Zeroth-order Optimization}\label{sec:prob_setup}
Given a function $F: \mathbb{R}^d \to \mathbb{R}$, and we consider the unconstrained minimization problem
\begin{equation}\label{eqn:problem}
x^* \in \argmin_{x \in \mathbb{R}^d} F(x).
\end{equation}

\noindent We assume that it admits at least one solution and we consider the setting where the gradient $\nabla F(x)$ is not accessible, and only function evaluations are available. We focus on finite-difference methods which are a class of algorithms that allow to tackle problem \eqref{eqn:problem} in this setting. These iterative procedures mimick first-order strategies by approximating the gradient of the target function $F$ with finite-differences on a set of directions. Formally, let $P \in \mathbb{R}^{d \times \ell}$, and denote by $(p^{(i)})_{i = 1}^{\ell}$ its columns. For a discretization parameter $h > 0$, we define the finite-difference gradient estimator as
\begin{equation}\label{eqn:forward_fd}
g(x, h, P) := \frac{d}{\ell} \sum\limits_{i=1}^\ell \frac{F(x + h p^{(i)}) - F(x)}{h} p^{(i)}.
\end{equation}
This estimator serves as a surrogate for $\nabla F(x)$ and is used to construct an iterative optimization procedure that mimics first-order methods. A common approach is to emulate the gradient descent algorithm. At each iteration $k \in \mathbb{N}$, let $\gamma_k, h_k > 0$ be the step size and the discretization parameter, respectively, and let $P_k \in \mathbb{R}^{d \times \ell}$ be a direction matrix. The iteration is then defined as
\begin{equation}\label{eqn:random_fd}
    x_{k+1} = x_k - \gamma_k g(x_k, h_k, P_k).
\end{equation}
These methods have been widely studied in the literature and can be broadly categorized into two classes based on the choice of the direction matrix: unstructured and structured methods. Unstructured methods rely on randomly sampled directions, typically drawn independently and identically distributed (i.i.d.) from a fixed distribution. Common choices include sampling independently from a standard Gaussian distribution \cite{nesterov2017random, duchi_power_of_two}, uniformly over the unit sphere \cite{yousefian2012stochastic, Berahas2022,gasnikov_sph}, or from a Rademacher distribution \cite{zoro, bszd}, where each entry is independently sampled from $\{-1, 1\}$ with equal probability. Structured methods, on the other hand, construct direction matrices that satisfy structural constraints such as orthogonality \cite{kozak2021zeroth, rando2022stochastic, rando2023optimal, str_zo_applied}. The intuition behind this choice is that imposing orthogonality reduces redundancy and improves local exploration, thereby yielding a more accurate gradient approximation.

\noindent The focus of this work is to investigate how introducing structure into the direction matrix $P$ affects the efficiency and accuracy of the finite-difference method when $\ell \leq d$. To this end, and to isolate the effect of the direction matrix, we introduce a finite-difference algorithm that select the stepsize using an adaptive line-search procedure. 

\begin{algorithm}[H]
\caption{Finite-Difference Line-search Method}\label{algo:ffd}
\begin{algorithmic}[1]
\For{$k = 0, \cdots$}
    \State $g_k = g(x_k, h_k, P_k)$
    \While{$F(x_k - \gamma g_k) > F(x_k) - c \gamma \|g_k\|^2$ and $\gamma > \gamma_\text{min}$}
        \State $\gamma \leftarrow \max\{\gamma \theta, \gamma_{\text{min}}\} $
    \EndWhile
    \If{$F(x_k - \gamma g_k) \leq F(x_k) - c \gamma \|g_k\|^2$}
       \State $x_{k + 1} = x_k - \gamma g_k $ and $\gamma \leftarrow \min\{\gamma \rho, \gamma_\text{max}\}$
    \EndIf
\EndFor
\end{algorithmic}
\end{algorithm}

\noindent Given an initial point $x_0 \in \mathbb{R}^d$, at each iteration $k \in \mathbb{N}$, the algorithm constructs a direction matrix $P_k \in \mathbb{R}^{d \times \ell}$ and estimates the gradient $g_k = g(x_k, h_k, P_k)$ using eq. \eqref{eqn:forward_fd}. It then checks the sufficient decrease condition $F(x_k - \gamma g_k) \leq F(x_k) - c \gamma \|g_k\|^2$, for a given Armijo condition constant $c \geq 0$. If the condition holds, the update $x_{k+1}$ is computed via eq. \eqref{eqn:random_fd}, and the stepsize is increased as $\gamma \leftarrow \min\{\gamma \rho, \gamma_{\max}\}$. Otherwise, the stepsize is reduced as $\gamma \leftarrow \max\{\gamma \theta, \gamma_{\min}\}$, until the condition is satisfied or $\gamma_{\min}$ is reached; in the latter case, the iterate is not updated. Note that \cite[Algorithm 3.1]{Cartis2018} can be recovered as a special case of Algorithm \ref{algo:ffd} with $\rho = 1/\theta$ and $\gamma_{\min} = 0$ . In the next section, we review and propose several strategies to construct structured direction matrices.

\subsection{Towards Efficient Structured Direction Generation}\label{sec:struct_directions}

Several works have proposed the use of structured direction matrices for finite-difference methods; see, e.g., \cite{kozak2021zeroth,rando2022stochastic,rando2023optimal,str_zo_applied,stief_zth,stief_zeroth,Berahas2022}. In this section, we review these approaches and propose several extensions. We denote by $I_{d,\ell} \in \mathbb{R}^{d \times \ell}$ the truncated identity matrix, and by $\tilde{I}_{d,\ell} \in \mathbb{R}^{d \times \ell}$ a random truncated permutation i.e. a matrix constructed by sampling $\ell$ columns from the identity matrix $I \in \mathbb{R}^{d \times d}$ without replacement. 

\paragraph*{QR \& Coordinate methods.} One strategy proposed in \cite{kozak2021zeroth,rando2022stochastic} constructs a structured direction matrix at each iteration $k\in \mathbb{N}$ by sampling a Gaussian matrix $A^{(k)} \in \mathbb{R}^{d \times \ell}$ with i.i.d. entries drawn from $\mathcal{N}(0,1)$, computing its QR factorization $A^{(k)} = Q^{(k)} R^{(k)}$, and setting $P_k = Q^{(k)}$. %
With slight modifications, this construction can yield a random matrix distributed according to the Haar measure \cite{mezzadri2006generate}. The method incurs a per-iteration computational cost of $\mathcal{O}(d \ell^2)$. Another strategy proposed in \cite{kozak2021zeroth,rando2022stochastic} uses a truncated random permutation matrix as the direction matrix, i.e., for every $k \in \mathbb{N}$, it set $P_k = \tilde{I}_{d,\ell}^{(k)}$. When $\ell = d$, this recovers the deterministic estimator of \cite{Berahas2022}. %

\paragraph*{Butterfly Directions.} An efficient method for constructing structured directions proposed in \cite{rando2023optimal}, leverages on Butterfly matrices \cite{TROGDON201948}. When $d = 2^n$ for some $n \geq 0$, a Butterfly matrix $G \in \mathbb{R}^{d\times d}$ is recursively defined as%
\begin{equation*}
G^{(n)} = \begin{bmatrix}
\cos \theta_n G^{(n-1)} & \sin \theta_n G^{(n-1)} \\
-\sin \theta_n G^{(n-1)} & \cos \theta_n G^{(n-1)}
\end{bmatrix}, \quad G^{(0)} = [1],
\end{equation*}
where $\theta_n \sim \mathcal{U}([0, 2\pi])$. At every iteration $k\in \mathbb{N}$, thus we can construct structured direction matrices as $P_k = G_k^{(n)} I_{d,\ell}$. This construction incurs a computational cost of $\mathcal{O}(d \log d)$ per iteration. However, it can be constructed only when $d$ is a power of two. %
To address this limitation, we propose to construct $G_k^{(n)}$ for the largest $n$ such that $2^n < d < 2^{n+1}$, then pad the matrix with rows and columns of the identity and sample $\ell$ random columns without replacement. %

\paragraph*{(Permuted) Householder Directions.} In \cite[Appendix D]{rando2023optimal}, a method that uses a random truncated Householder reflector is proposed. At each iteration $k$, a random vector $v_k$ is sampled uniformly from the unit sphere, and the direction matrix is constructed as $P_k = (I - 2 v_k v_k^\top) I_{d,\ell}$. This method is memory-efficient, as we only need to store the vector $v_k$ in memory, and it incurs a computational cost of $\mathcal{O}(d \ell)$ (since we only need to compute $\ell$ columns of the outer product). However, in our experiments, we observed that this strategy results in poor convergence performance when $\ell < d$. To address this, we propose a variant that improves performance without increasing the computational complexity by replacing the truncation with a random truncated permutation. Thus, at each iteration $k$, the direction matrix is constructed as $P_k = (I - 2 v_k v_k^\top) \tilde{I}_{d,\ell}^{(k)}$. 
This variant retains low computational overhead and, as shown in Section \ref{sec:results}, achieves significantly better performance when $\ell \ll d$.

\paragraph*{Other Methods.} In \cite{stief_zth, stief_zeroth}, the authors propose structured finite-difference algorithms where, at each iteration $k$, a Gaussian matrix $A_k \in \mathbb{R}^{d \times \ell}$ is sampled, and the direction matrix is computed as $P_k = A_k (A_k^\top A_k)^{-1/2}$. These matrices are structured, as they lie on the Stiefel manifold \cite{chikuse2012statistics}. However, this approach is computationally expensive, incurring a per-iteration cost of $\mathcal{O}(d \ell^2)$. In \cite{str_zo_applied}, a method that constructs structured direction matrices as the product of Rademacher-Hadamard matrices is proposed and used. While this approach is computationally efficient, it shares the same limitation as Butterfly directions - it can only be applied when $d$ is a power of 2. To address this, the authors propose padding the direction matrices with zeros, but this can lead to stagnation as some coordinates remain unchanged.

\subsection{Preliminary Observations}\label{sec:preliminary_obs}
In this section, we discuss the expected outcomes based on both empirical observations and theoretical findings in the literature. According to theoretical analyses of finite-difference methods \cite{nesterov2017random, kozak2021zeroth, rando2022stochastic}, the use of structured directions does not affect the dominant terms in the convergence rate but they typically influence the gradient approximation error and lower-order terms. As a motivating example, consider directions $(p^{(i)})_{i =1}^\ell$ sampled from the unit sphere. If $p^{(i)}$ are i.i.d then the expected gradient approximation error is
\begin{equation}\label{eqn:approx_error_sph}
\mathbb{E}[\| g(x, h, P) - \nabla F(x) \|^2] = \mathbb{E}[\| g(x, h, P) \|^2] + \|\nabla F(x)\|^2 + C_1(x, h),
\end{equation}
where $x \in \mathbb{R}^d$, $h > 0$, $C_1(x,h)$ is a constant depending on $x$ and $h$. In this case, we have that
\begin{equation*}
\mathbb{E}[\| g(x, h, P) \|^2] = \frac{d^2}{\ell^2 h^2} \left[ \sum_{i = 1}^\ell \mathbb{E}(F(x + h p^{(i)}) - F(x))^2 + \ell(\ell - 1) C_2(x, h) \right],
\end{equation*}
where $C_2(x, h) \geq 0$ is a constant capturing the interaction between directions. If $p^{(i)}$ are orthogonal directions sampled according to the Haar measure \cite{mattila_1995} then %
we can show that we get a smaller approximation error
\begin{equation*}
\mathbb{E}[\| g(x, h, P) - \nabla F(x) \|^2] = \frac{d^2}{\ell^2 h^2} \sum_{i = 1}^\ell \mathbb{E}(F(x + h p^{(i)}) - F(x))^2 + \|\nabla F(x) \|^2 + C_1(x, h).
\end{equation*}
Formal derivations and additional discussion are provided in Appendix~\ref{app:str_unstr}.
This variance reduction can result in more stable optimization trajectories and faster empirical convergence \cite{choromanski2019unifying, Berahas2022}.
Moreover, small variance is essential for line-search procedures. As shown in \cite{Cartis2018}, finite-difference methods that incorporate line-search require the gradient surrogate to be a sufficiently accurate approximation of the true gradient with high probability in order to guarantee convergence. The work \cite{Berahas2022} establishes that such a high-probability accuracy condition is satisfied by unstructured estimators when the number of directions exceeds the dimensionality, i.e., $\ell > d$. Moreover, the same condition holds for estimators based on $\ell = d$ coordinate directions or interpolation-based schemes. The latter refers to constructing the gradient surrogate by fitting a local linear model over a set of interpolation points. Notice that when $\ell = d$ and the directions are orthogonal, such a surrogate coincides with eq. \eqref{eqn:forward_fd}, which enables extending the results of \cite{Berahas2022} to structured estimators using orthogonal directions. However, no similar theoretical guarantees have been established in the setting where $\ell < d$. In this regime, the most closely related results are those of \cite{kozak2021zeroth,rando2022stochastic}, which show that it is possible to bound the error between the structured gradient surrogate and the projection of the true gradient onto the span of the direction matrix, i.e., $\|g(x, h, P) - P P^\top \nabla F(x)\|$. These results suggest that, for sufficiently large $\ell < d$, structured directions can still yield reliable approximations.
Moreover, the variance can be further reduced by increasing the number of directions $\ell$. %

\section{Experimental Setting}\label{sec:perf_metrics}

In this section, we describe the performance metrics and the benchmark functions used in our ablation study. We implement Algorithm \ref{algo:ffd} and investigate how the choice of the direction matrix influences the overall optimization performance. 
We consider a diverse set of direction generation procedures. For unstructured methods, we consider the three most common methods described in Section \ref{sec:prob_setup}: Gaussian, spherical and Rademacher while for structured methods, we consider the QR method, the random coordinate strategy, the single random Householder reflector, the random permuted Householder method and the Butterfly method with permutation discussed in Section \ref{sec:struct_directions}.

\subsection{Performance Metrics}
\label{sec:metrics}
We assess performance across three main criteria: the computational cost of generating direction matrices, the gradient approximation error, and the function value progress during optimization.

\paragraph*{Computational Cost.}
We measure the time required to construct the direction matrix $P \in \mathbb{R}^{d \times \ell}$ for each method, under the assumption that sufficient memory is available to store the full matrix. Our aim is to show that, even when generated explicitly, several structured methods incur a computational cost comparable to unstructured ones. Moreover, many structured schemes can be implemented in a memory-efficient manner, making them suitable for high-dimensional settings. We emphasize that the focus of this work is not on implementing large-scale structured finite-difference methods, but rather on empirically evaluating the benefits of imposing structure in the choice of directions. The results of this study motivates the development of scalable implementations of such methods. %

\paragraph*{Gradient Approximation Error.}
Let $x \in \mathbb{R}^d$ be such that $\nabla F(x) \neq 0$, and let $h > 0$. For a direction matrix $P \in \mathbb{R}^{d \times \ell}$, we define the (relative) gradient approximation error $E(x, h, P)$ as
\begin{equation*}
    E(x, h, P) := \frac{\| g(x, h, P) - \nabla F(x) \|}{\|\nabla F(x)\|},
\end{equation*}
where $g(x, h, P)$ is the gradient surrogate defined in eq. \eqref{eqn:forward_fd}. This metric quantifies the fidelity of the gradient estimator and has been used in prior work \cite{Berahas2022}.

\paragraph{Function Value Progress.}
To study the impact of the direction generation strategies on optimization performance, we measure the (normalized) function value progress over time. Let $x_0 \in \mathbb{R}^d$ be the initial guess, and let $x_k \in \mathbb{R}^d$ be the iterate produced by iteration \ref{eqn:random_fd}. We define the progress metric as
\begin{equation*}
    V(x_k, x_0) := \frac{F(x_k) - \min F}{F(x_0) - \min F}.
\end{equation*}
In experiments where $\min F$ is not known, we replace it with the smallest function value observed across all methods. %

\paragraph*{Performance Summaries.}
To summarize performance across a large collection of problems, we use the fraction of problems solved as a comparative metric across different problem classes, based on the criteria introduced previously. Specifically, for evaluating gradient approximation quality, we consider a problem to be "solved" if the expected relative gradient approximation error falls below a user-specified threshold $\tau > 0$. Similarly, to evaluate optimization performance, we consider a problem "solved" if the expected relative reduction in function value measured at the best iterate, (i.e., the iteration at which the lowest function value is obtained) is below a threshold $\tau$. Formally, let $n \in \mathbb{N}$ denote the total number of problems in the test suite $\{F_i\}_{i=1}^n$. Then, the performance metrics are defined as
\begin{equation*}
\rho_\text{grad}(\tau) := \frac{1}{n} \sum_{i = 1}^n 1_{\mathbb{E}_P \left[ \frac{\| g(x_i, h, P) - \nabla F_i(x_i) \|}{\|\nabla F_i(x_i)\|} \right] \leq \tau} \qquad \rho_\text{val}(\tau) := \frac{1}{n} \sum_{i = 1}^n 1_{ \mathbb{E}_P \left[ \frac{F_i(x_k^i) - \min F_i}{F_i(x_0^i) - \min F_i} \right] \leq \tau},
\end{equation*}
where $x_i$ denotes the evaluation point for the $i$-th problem, $x_0^i$ and $x_k^i$ denote the initial and best-found points respectively, and $\min F_i$ is an estimate of the minimal value of the $i$-th objective function. The expectation over $P$ is approximated using a fixed number of direction matrix samples. Similar methods to summarize results were used and proposed in \cite{more_wild,dolan_more}. %

\subsection{Benchmark Problems}
To evaluate the performance of structured versus unstructured directions in finite-difference methods, we consider three distinct types of benchmark problems: synthetic functions, real-world optimization problems from the CUTEst collection, and a high-dimensional adversarial attack task on image classification. The synthetic functions we considered are commonly used for benchmarking zeroth-order optimization methods and they are the Least squares function, Qing Function and Rosenbrock function. We fix the input dimension $d = 500$ for all synthetic tests and use randomly generated inputs with a consistent seed to ensure reproducibility. To assess performance on more diverse and practical problems, we select a subset of the unconstrained optimization problems from the CUTEst benchmark suite \cite{Gould2015} - see Appendix \ref{app:exp_details} for additional details. The selected problems span input dimensions ranging from $6$ to $1000$ and have no constraints or discontinuities. For each problem, we use the standard initial point provided by the benchmark and run the algorithms for $10000$ function evaluations. Finally, we evaluate direction strategies on a high-dimensional, real-world task of generating adversarial examples for a multiclass classifier trained on MNIST \cite{mnist_dataset}. Such a classifier is a convolutional neural network (CNN) trained using defensive distillation~\cite{defensive_distillation}. The attack is cast as an optimization problem where the objective is to perturb the input image to induce misclassification while minimizing the perturbation norm. Formally, let $f_w : \mathbb{R}^d \rightarrow \mathbb{R}^C$ be a classifier parameterized by weights $w$, mapping an input image $z \in \mathbb{R}^d$ to a probability distribution over $C$ classes via the softmax output. We consider a fixed input-label pair $(z, y)$ from the MNIST test set correctly classified by the network $f_w$, where $z \in \mathbb{R}^d$ is the original image (normalized to be in $[-0.5, 0.5]^d$) and $y \in \{1, \ldots, C\}$ is the true label. We adopt the same
black-box attacking loss as in \cite{ji_var_red,zoo,liu_svr}
\begin{equation*}
    \min\limits_{x \in \mathbb{R}^d} F(x) := \max\left\{ \log f_w(\psi(x, z) )_y - \max\limits_{j \neq y} \log f_w(\psi(x, z))_j, -\kappa \right\} + \frac{\lambda}{2} \|\psi(x, z) - z\|^2,
\end{equation*}
where $f_w(\cdot)_j$ denotes the predicted probability for class $j$, $\lambda > 0$ is a regularization parameter and $\psi$ is the manipulation function defined as $\psi(x, z) = 0.5 \tanh{(\tanh^{-1}(2z) + x)}$.
For this experiment, we evaluate each method under a fixed budget of $30000$ function evaluations, aimed at generating adversarial perturbations for $100$ MNIST images randomly sampled from the test set. Performance is reported in terms of the fraction of problems successfully solved, where each problem corresponds to the minimization of the attack loss for a distinct image. Full experimental details, including the model architecture, are provided in Appendix~\ref{app:exp_details}.

\section{Results}\label{sec:results}

In this section we perform our ablation study. We first analyze the computational cost of generating direction matrices, then assess gradient approximation accuracy and function value progress. Experiments span both synthetic and real-world problems to highlight the practical benefits of structured directions.

\paragraph*{Computational Cost.} We evaluate the computational cost of generating structured and unstructured direction matrices across varying numbers of directions $\ell$ and input dimensions $d$. In Figure \ref{fig:time_cost}, we report the mean and standard deviation of the time required to generate such matrices using $500$ repetitions. The result indicates that, except for QR method, structured generation schemes considered achieve comparable time to unstructured  methods, even in high-dimensional regimes. Notably, for coordinate directions with $\ell = d$, the identity matrix can be precomputed and reused throughout the algorithm. This eliminates the need to generate new matrices at each iteration, leading to a substantial reduction in per-iteration overhead.
\begin{figure}[H]
    \centering
    \includegraphics[width=0.9\linewidth]{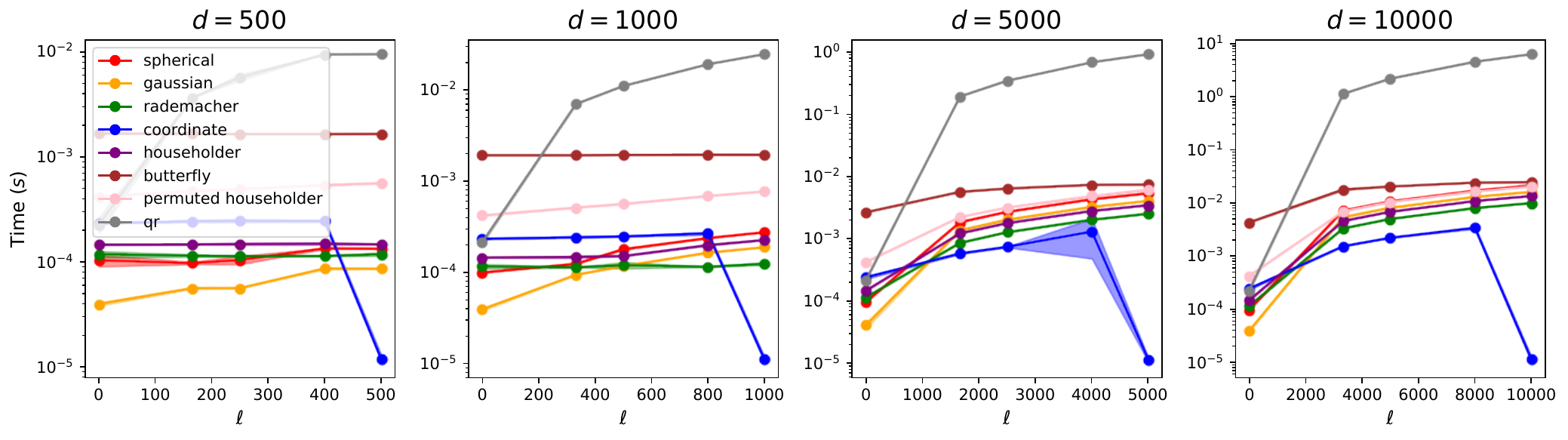}
    \caption{Time cost for constructing direction matrices.}
    \label{fig:time_cost}
\end{figure}
\paragraph*{Gradient Approximation Error.} We compare the gradient approximation error using the gradient surrogate defined in eq. \eqref{eqn:forward_fd}. In Figure \ref{fig:synthetic_functions}, we report the relative error in computing the gradient approximation for multiple synthetic test functions, over $50$ trials. The input dimension is fixed at $d = 500$, the finite-difference discretization parameter sequence is constant and it is set to $h_k = 10^{-7}$, and the number of directions $\ell$ is selected as a fraction of $d$. For $\ell \geq d/3$, structured methods provide better performance than the unstructured ones consistently and, in particular, when $\ell = d$, structured estimators outperform their unstructured counterparts. These results extend the findings of \cite{Berahas2022}, which demonstrated similar improvements primarily for coordinate-based or interpolation-based schemes in setting $\ell \geq d$. This improvement can be attributed to the lower variance induced by structured directions, which tend to span the space more uniformly and avoid the redundancy often present in random directions. This can be also justified theoretically by observing that for $\ell > 1$, the approximation error for structured direction estimator is smaller than the unstructured. In Figure \ref{fig:grad_acc_cutest}, we further validate these trends on the subset of CUTEst benchmark problems, reporting the fraction of solved problems for each method changing the accuracy threshold $\tau$. A method is better than another if it solves an higher fraction of problems for a smaller $\tau$ which indicates, in this case, the gradient accuracy. We can observe that for $\ell = d/3$, structured methods provide slightly better performance. Moreover, for $\ell \geq d/2$, structured estimators exhibit evident higher performance and for $\ell =d$ structured methods outperform unstructured ones. %

\begin{figure}[H]
    \centering
    \includegraphics[width=0.9\linewidth]{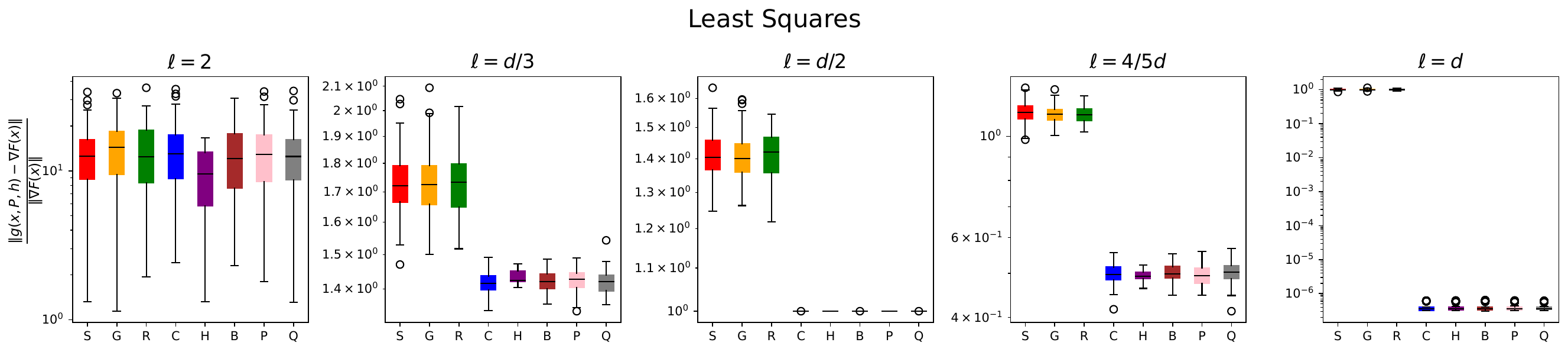}
    \includegraphics[width=0.9\linewidth]{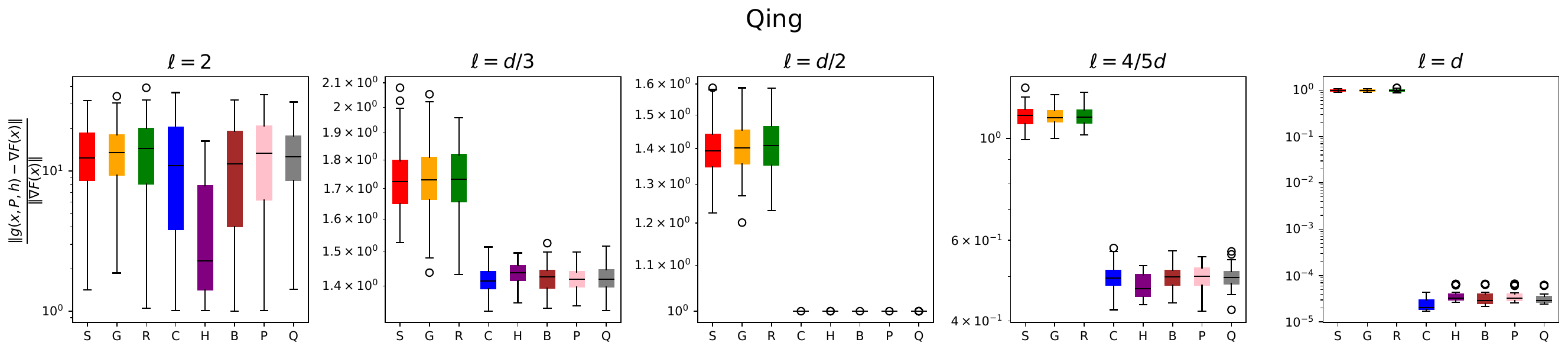}
    \includegraphics[width=0.9\linewidth]{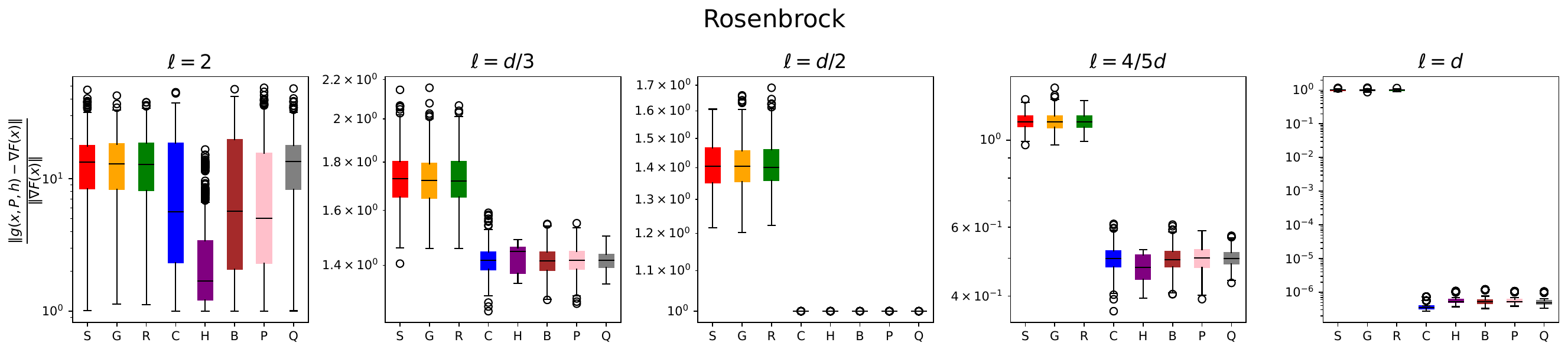}
    \caption{ Relative gradient approximation error for the Least Squares, Qing, and Rosenbrock functions using the surrogate in Eq.~\eqref{eqn:forward_fd}, with direction matrices generated by: S (Spherical), G (Gaussian), R (Rademacher), C (Coordinate), H (Householder), B (Butterfly), P (Permuted Householder), Q (QR).
    }
    \label{fig:synthetic_functions}
\end{figure}
\begin{figure}[H]
    \centering
    \includegraphics[width=0.9\linewidth]{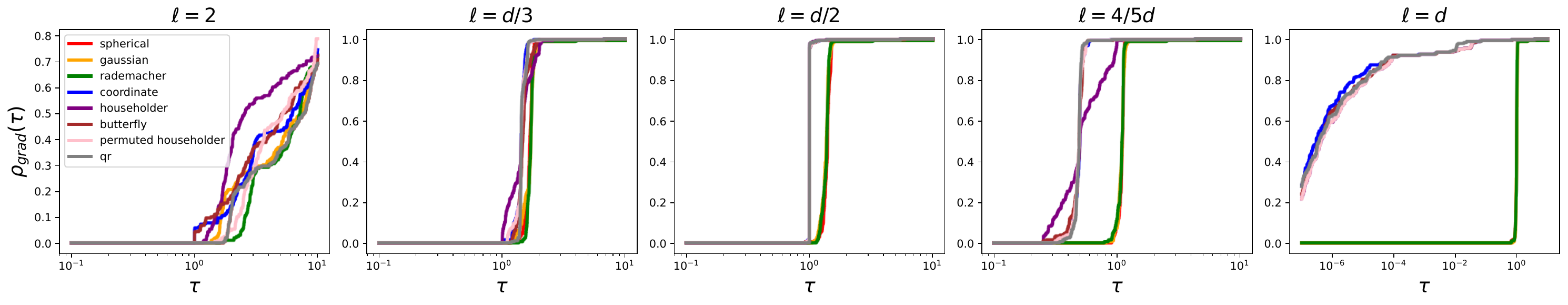}
    \caption{Fraction of solved problems for gradient approximation error on subset of CUTEst benchmark.}
    \label{fig:grad_acc_cutest}
\end{figure}

\paragraph*{Convergence.} 
Here, we compare the performance of Algorithm \ref{algo:ffd} in optimizing several synthetic objective functions using different methods for constructing direction matrices. %
We fix a total budget of $10000$ function evaluations, and terminate each run once this budget is exceeded. Each experiment is repeated $10$ times, and we report the mean and standard deviation of the objective values at termination. Results are summarized in Figure~\ref{fig:ls_conv}. Our findings show that, with the exception of the standard Householder method, structured approaches consistently performs better than unstructured ones when $\ell \geq d/3$. For $\ell = d/3$, structured methods yield performance that is comparable to (or slightly better than) their unstructured counterparts. The advantage becomes more pronounced when $\ell > d/2$, reflecting the improved gradient accuracy observed in earlier sections. Interestingly, for $\ell = d$, performance slightly degrades compared to $\ell = d/2$ on some functions. This can be attributed to the fixed budget: when $\ell = d$, each iteration is more expensive, resulting in fewer total iterations and updates within the evaluation limit. We observe also that for $\ell > 2$, structured methods performs better than the case $\ell = 2$. This is likely due to the use of multiple directions, which yields more accurate gradient approximations and enables the line search to select larger step sizes. This effect is supported by theoretical insights \cite{rando2022stochastic, rando2023optimal}, which show that to guarantee convergence the stepsize depends on the number of directions. Finally, when considering the computational cost of each method (as discussed in the previous section), it is worth noting that using multiple directions can also be more efficient in practice (in particular when $d$ is not too large). Although the per-iteration cost increases with $\ell$, methods that rely on only a few directions often require more iterations with less informative updates. Moreover, once the direction matrix $P$ is generated, the associated function evaluations can be parallelized, thereby mitigating the computational overhead.%
\begin{figure}[t]
    \centering
    \includegraphics[width=0.9\linewidth]{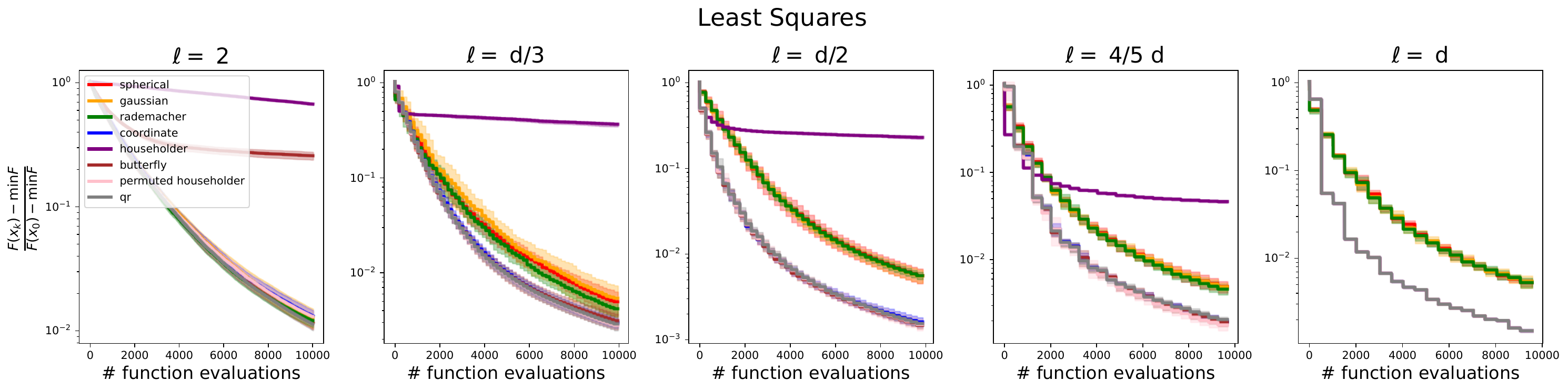}
    \includegraphics[width=0.9\linewidth]{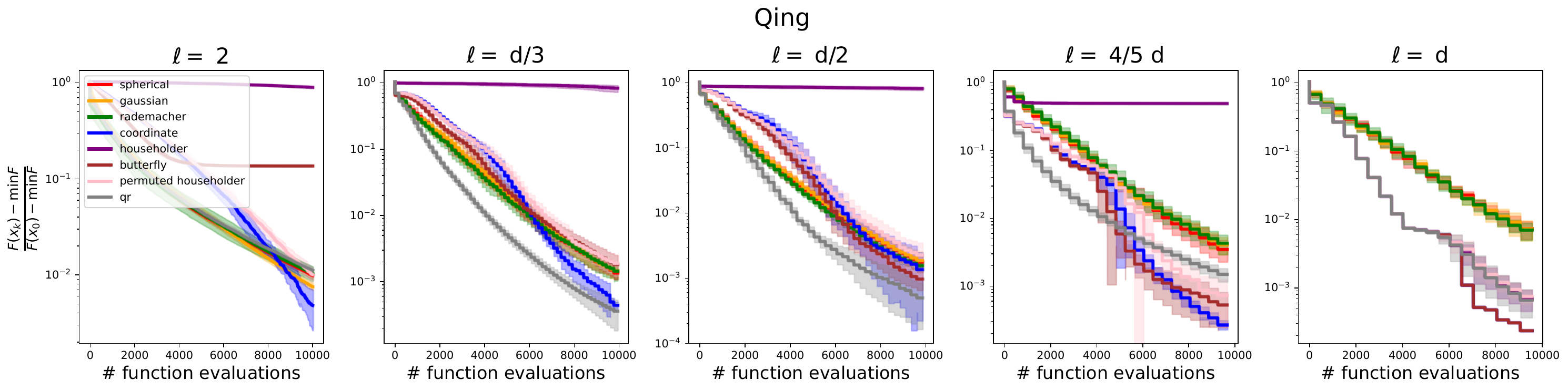}
    \includegraphics[width=0.9\linewidth]{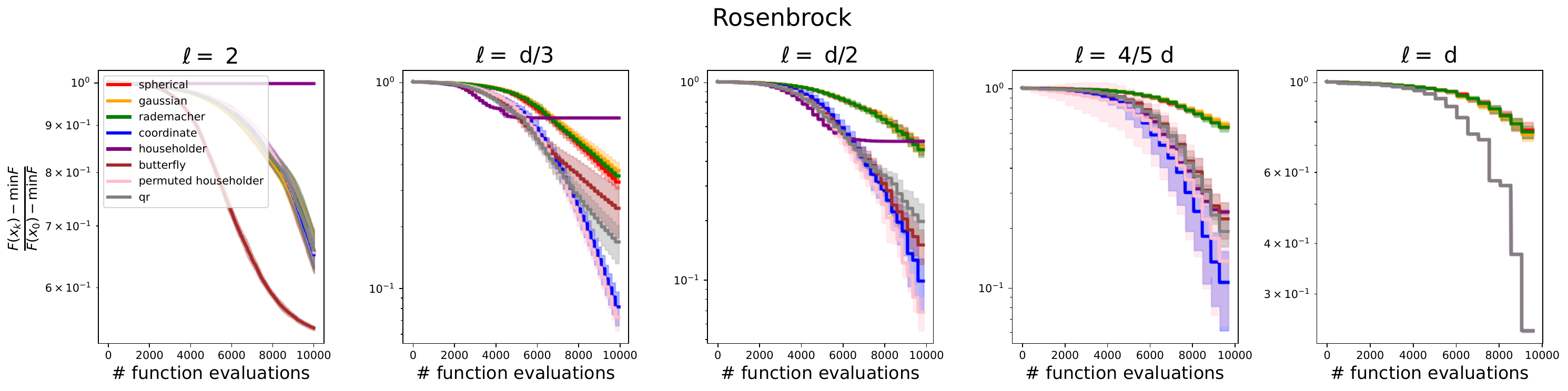}
    \caption{Function value progress in optimizing Least-square, Qing and Rosenbrock functions.}
    \label{fig:ls_conv}
\end{figure}
\begin{figure}[t]
    \centering
    \includegraphics[width=0.9\linewidth]{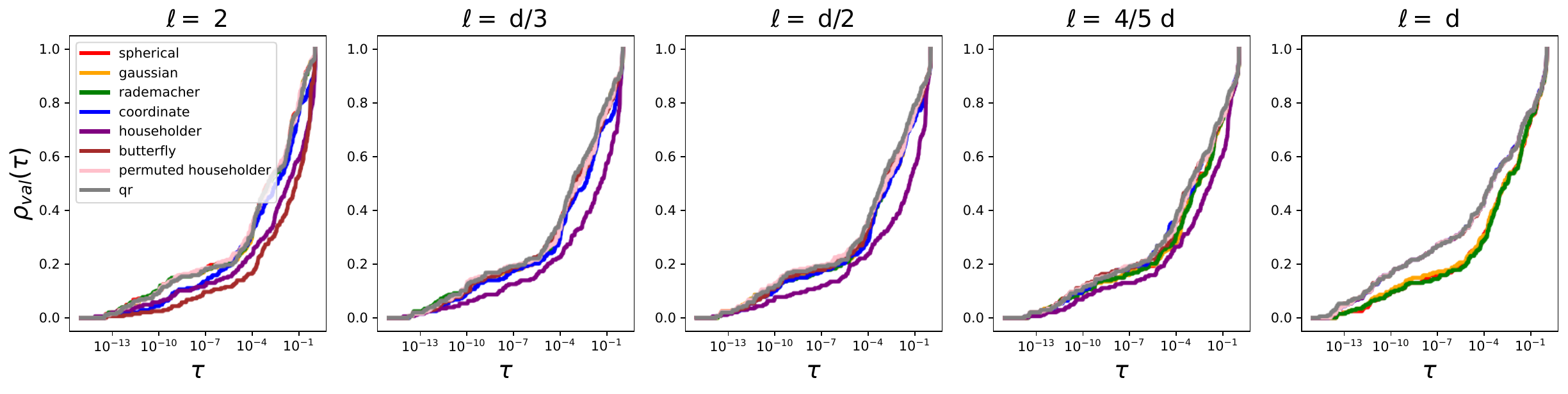}\\
    \includegraphics[width=0.9\linewidth]{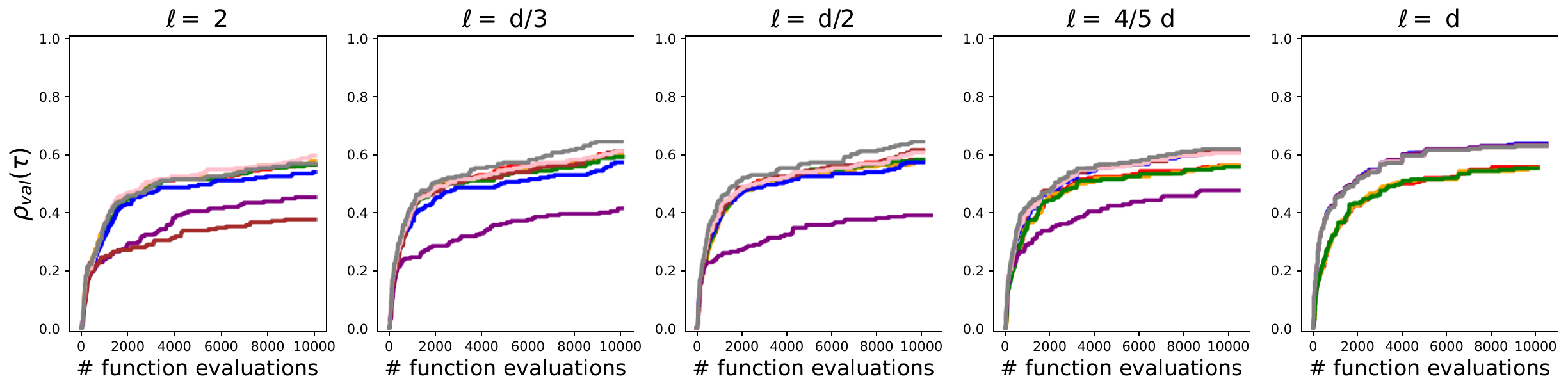}
    \caption{Top row: Fraction of problems solved as a function of accuracy threshold $\tau$.
Bottom row: Fraction of problems solved at fixed accuracy $\tau = 10^{-2}$ versus the function evaluation budget.}
    \label{fig:frac_solved_cutest_conv}
\end{figure}
\noindent In the first row of Figure \ref{fig:frac_solved_cutest_conv}, we show the fraction of problems solved for the CUTEst benchmark. We observe that, in general, with the exception of the single Householder method, all approaches solve approximately the same fraction of problems. Notably, for $\ell = 2$, the coordinate and butterfly direction methods exhibit worse performance compared to the others. This behavior may be due to structural limitations in the directions used. In particular, coordinate directions can be ineffective when the objective function is insensitive to certain variables (i.e., when some dimensions have little or no influence on the function value) resulting in search directions that contribute minimally to progress. Similarly, butterfly directions are affected by padding: since the domain dimension is not necessarily a power of two, additional rows and columns from the identity matrix are introduced, as described in Section \ref{sec:struct_directions}. In the second row of Figure \ref{fig:frac_solved_cutest_conv}, we fix the accuracy level $\tau = 10^{-2}$ and we plot the fraction of solved problem per number of function evaluations.  We can observe that for $\ell \geq d/2$ (except for Householder and Coordinate), structured methods provide an higher fraction of solved problems before the unstructured methods, suggesting that such methods allow to reach faster an approximate solution as observed in the synthetic experiments (Figure \ref{fig:ls_conv}). For higher values of $\ell$ the result becomes more pronounced, and for $\ell = d$ every structured methods performs better than unstructured ones. 
In Figure \ref{fig:fvals_adv}, we report the fraction of problems solved with a threshold $\tau = 0.5$ for the adversarial perturbation task. The results show that finite-difference methods with structured directions consistently outperform their unstructured counterparts when $\ell \geq d/3$. Notably, for $\ell = 2$, none of the methods succeed in solving any of the instances under the given criteria. This does not imply that it is impossible to generate a misclassifying adversarial perturbation with $\ell = 2$ and a budget of $30000$ function evaluations; rather, it suggests that, with these budget and parameters, the methods fail to produce a solution whose function value is sufficiently close to the best observed across all trials and thus, an higher budget could be required. %

\begin{figure}[H]
    \centering
    \includegraphics[width=0.9\linewidth]{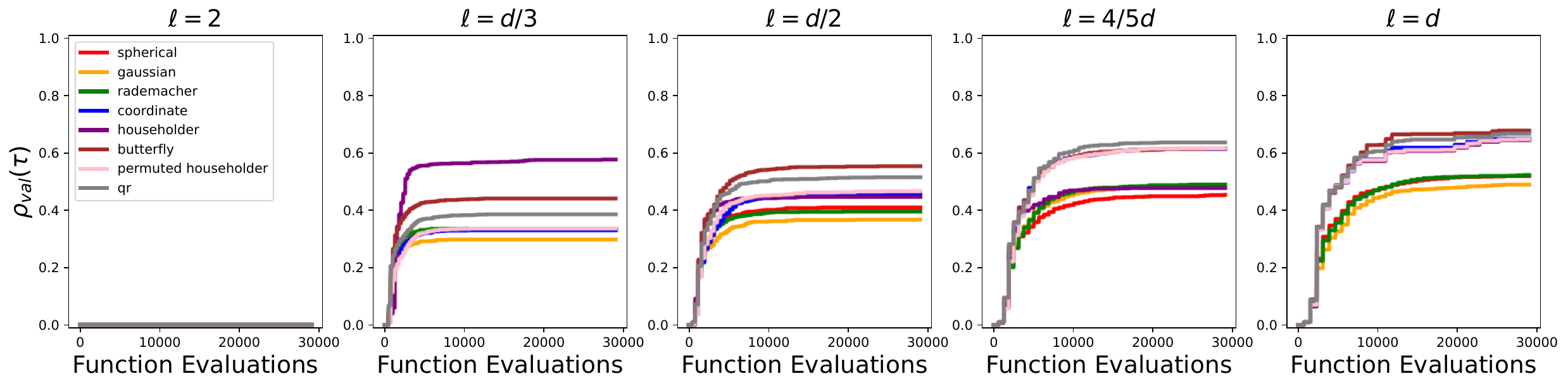}
    \caption{Fraction of solved problems with accuracy $\tau = 0.5$}
    \label{fig:fvals_adv}
\end{figure}

\section{Conclusions}\label{sec:conclusion}
In this work, we investigated the impact of structured versus unstructured directions in finite-difference optimization. We introduced several structured strategies and compared them to common unstructured methods. Our results show that structured directions can match the efficiency of unstructured ones while yielding more accurate gradient estimates, leading to faster convergence and more accurate solutions when used in Algorithm \ref{algo:ffd}. This work highlights the value of structure in direction matrices and opens several avenues for future research. These include theoretical analysis of Algorithm~\ref{algo:ffd} with structured surrogates, extending the study to other estimators such as central differences or interpolation-based methods \cite{Berahas2022}, and exploring applications of structured directions to memory-efficient finite-difference algorithms for large language model (LLM) fine-tuning such as \cite{mezo}.

\paragraph*{Acknowledgments.} L. R. acknowledges the financial support of the European Commission (Horizon Europe grant ELIAS 101120237), the Ministry of Education, University and Research (FARE grant ML4IP R205T7J2KP) and the Center for Brains, Minds and Machines (CBMM). S. V. acknowledges the support of the European Commission (grant TraDE-OPT 861137). L. R. and S. V. acknowledge the financial support of the Ministry of Education, University and Research (grant BAC FAIR PE00000013 funded by the EU - NGEU). M. R., L. R. and S. V. acknowledge the financial support of the European Research Council (grant SLING 819789). L. R., S. V. and C. M. acknowledge the financial support of the US Air Force Office of Scientific Research (FA8655-22-1-7034). The research by S. V. and C. M. has been supported by the MUR Excellence Department Project awarded to Dipartimento di Matematica, Universita di Genova, CUP D33C23001110001 and MIUR (PRIN 202244A7YL). M. R., C. M. and S. V. are members of the Gruppo Nazionale per l’Analisi Matematica, la Probabilità e le loro Applicazioni (GNAMPA) of the Istituto Nazionale di Alta Matematica (INdAM). This work represents only the view of the authors. The European Commission and the other organizations are not responsible for any use that may be made of the information it contains.

\bibliographystyle{plain}
\bibliography{bibliography}

\appendix
\newpage

\section{Experimental Details}\label{app:exp_details}
In this appendix, we provide additional details on the experimental setup. All scripts were implemented in Python 3 (version 3.10.11) using the following libraries: NumPy (version 1.24.3) \cite{harris2020array}, Matplotlib (version 3.5.1) \cite{matplotlib}, and PyTorch (version 2.0.1) \cite{pytorch}. For the optimization benchmarks, we employed the PyCUTEst interface \cite{PyCUTEst2022} to access a subset of the CUTEst test functions. The adversarial perturbation experiments were conducted on the MNIST dataset \cite{mnist_dataset}.

\paragraph*{Machine used to perform the experiments.} In the following table, we describe the features of the
machine used to perform the experiments in Section \ref{sec:results}.
\begin{table}[H]
    \centering
    \caption{Machine used to perform the experiments}\label{tab:machine}
    \begin{tabular}{l l}
        \toprule
        Feature & \\ 
        \midrule
         OS & Ubuntu 18.04.1\\
         RAM & 256 GB \\
         GPU(s)& 1× Quadro RTX 6000 (24 GB)\\
         CUDA version & 11.7\\
         \bottomrule
    \end{tabular}
\end{table}

\subsection{Synthetic \& CUTEst experiment details}
In this appendix, we report the details of the synthetic and CUTEst benchmark experiments performed in Section~\ref{sec:results}.

\paragraph*{Objective Functions.} We evaluated algorithmic performance on three standard synthetic functions: Least Squares, Qing, and Rosenbrock, whose definitions are reported in Table~\ref{tab:synt_functions}. In particular, the Least Squares function is given by
\begin{equation*}
F(x) = \frac{1}{2} \|Ax - y\|^2,
\end{equation*}
where $x \in \mathbb{R}^d$, $A \in \mathbb{R}^{d \times d}$, $x^* \sim \mathcal{N}(0, I)$, and $y = Ax^*$. The matrix $A$ is constructed to control the conditioning of the problem. Specifically, we generate a random Gaussian matrix $\bar{A}$ with i.i.d. entries from $\mathcal{N}(0, 1)$, compute its QR factorization $\bar{A} = QR$, and define $A = Q S Q^\top$, where $S$ is a diagonal matrix with entries linearly spaced between $\sqrt{\mu}$ and $\sqrt{L}$. This construction ensures that $F$ is $\mu$-strongly convex and has an $L$-Lipschitz continuous gradient. In all experiments of Section \ref{sec:results}, we set $L = 10^4$ and $\mu = 1$, and the dimension of the input space is $d=500$. For both convergence and gradient approximation experiments, we use $x_0 = [1, \dots, 1]$ for the Least Squares and Qing functions, and $x_0 = [0.5, \dots, 0.5]$ for the Rosenbrock function. 

\noindent For the experiments on the CUTEst test suite, we selected a subset of unconstrained problems with input dimensions ranging from $6$ to $1000$. For each problem, the initial condition $x_0 \in \mathbb{R}^d$ was set using the default initialization provided by the PyCUTEst library~\cite{PyCUTEst2022}. The selected problems are listed in Table~\ref{tab:cutest_problems}, grouped by function name according to the CUTEst classification, along with the corresponding input dimensions used in our experiments.

\begin{table}[H]
    \centering
    \caption{Definitions of the synthetic benchmark functions used in the experiments.}
    \label{tab:synt_functions}
    \begin{tabular}{ll}
        \toprule
         Name & Definition \\
         \midrule
         Least Squares & $F(x) =\frac{1}{2}\|Ax - y\|^2$ \\
         Qing & $F(x) =\sum\limits_{i = 1}^d (x_i^2 - i)^2$\\
         Rosenbrock & $F(x) = \sum\limits_{i = 1}^{d- 1} 100(x_{i + 1} - x_i^2)^2 +(x_i - 1)^2$\\
         \bottomrule
    \end{tabular}

\end{table}

\begin{table}[H]
\centering
\caption{Names and input dimensions ($d$) of the selected problems from the CUTEst benchmark suite.} \label{tab:cutest_problems}
\begin{tabular}{llll}
	\toprule
	Problem & $d$ & Problem & $d$ \\
	\midrule
	ARGLINA & $10, 50, 100, 200$ & INDEF & $10, 50, 100, 1000$ \\
	ARGTRIGLS & $10, 50, 100, 200$ & INDEFM & $10, 50, 100, 1000$ \\
	ARWHEAD & $100, 500, 1000$ & INTEQNELS & $12, 52, 102, 502$ \\
	BDQRTIC & $100$ & LIARWHD & $100$ \\
	BOX & $10, 100, 1000$ & MOREBV & $10, 50, 100, 500, 1000$ \\
	BOXPOWER & $10, 100, 1000$ & NCB20 & $110$ \\
	BROWNAL & $10$ & NCB20B & $50, 100, 180, 500, 1000$ \\
	BROYDN3DLS & $10, 50, 100, 1000$ & NONCVXU2 & $10$ \\
	BROYDNBDLS & $10, 50, 100, 500, 1000$ & NONCVXUN & $10$ \\
	BRYBND & $10, 50, 100, 500, 1000$ & NONDIA & $10, 20, 30, 50, 90, 100$ \\
	CHNROSNB & $10, 50$ & NONDQUAR & $100, 500, 1000$ \\
	CHNRSNBM & $10, 50$ & OSCIPATH & $10, 100, 500$ \\
	COSINE & $10, 100, 1000$ & PENALTY2 & $10$ \\
	CURLY10 & $100, 1000$ & POWELLSG & $8, 20, 40, 60, 80, 100, 500, 1000$ \\
	CURLY20 & $100, 1000$ & POWER & $10, 20$ \\
	CURLY30 & $100, 1000$ & SBRYBND & $10, 50, 100, 500, 1000$ \\
	DIXON3DQ & $10, 100, 1000$ & SCHMVETT & $10, 100, 500, 1000$ \\
	DQDRTIC & $10, 50$ & SCOSINE & $10, 100, 1000$ \\
	DQRTIC & $10$ & SENSORS & $10, 100$ \\
	EIGENALS & $6, 110$ & SINQUAD & $50, 100, 500, 1000$ \\
	EIGENBLS & $6, 110$ & SINQUAD2 & $50, 100, 500, 1000$ \\
	ENGVAL1 & $50, 100$ & SPARSINE & $10, 50, 100$ \\
	ERRINROS & $10$ & SPARSQUR & $10, 50, 100$ \\
	ERRINRSM & $10$ & SPIN2LS & $6, 8, 12, 22, 42, 62, 82, 102$ \\
	EXTROSNB & $10, 100$ & SPINLS & $7, 11, 22, 67, 232, 497, 862$ \\
	FLETBV3M & $10, 100, 1000$ & SSBRYBND & $10, 50, 100, 500, 1000$ \\
	FLETCBV2 & $10, 100, 1000$ & SSCOSINE & $10, 100, 1000$ \\
	FLETCBV3 & $10, 100, 1000$ & STRTCHDV & $10, 100, 1000$ \\
	FLETCHBV & $10, 100, 1000$ & TOINTGSS & $10, 50, 100, 500, 1000$ \\
	FLETCHCR & $10, 100, 1000$ & TQUARTIC & $10, 50, 500, 1000$ \\
	FREUROTH & $10, 50, 100$ & TRIDIA & $10, 20, 30, 50, 100$ \\
	GENROSE & $10, 100, 500$ & TRIGON1 & $10, 100$ \\
	HILBERTA & $6, 10$ & TRIGON2 & $10, 100, 1000$ \\
	HILBERTB & $10, 50$ &  &  \\
	\bottomrule
\end{tabular}
\end{table}

\paragraph*{Optimization parameter settings.} For all experiments on synthetic functions, we run Algorithm~\ref{algo:ffd} using the same line-search configuration across all methods to ensure fair comparison. Specifically, we adopt an initial step size of $\gamma_0 = 1.0$, and fix the Armijo condition constant at $c = 10^{-7}$. The minimum and maximum step sizes are set to $\gamma_{\text{min}} = 10^{-10}$ and $\gamma_{\text{max}} = 1000$ respectively. For the experiments on the selected subset of the CUTEst benchmark suite, we set the initial step size to $\gamma_0 = 0.5$, and the Armijo condition constant to $c = 10^{-5}$. In this case, the minimum and maximum step sizes are fixed to $\gamma_{\text{min}} = 10^{-10}$ and $\gamma_{\text{max}} = 1.0$. In both settings, we use the same expansion and contraction factors for the line-search procedure, namely $\rho = 2.0$ and $\theta = 0.5$, respectively. The finite-difference discretization parameter is kept constant across all iterations and is set to $h_k = 10^{-7}$.

\subsection{Adversarial Perturbation Experiment Details}

In this appendix, we report the implementation and training details of the adversarial perturbation experiments presented in Section~\ref{sec:results}, along with the architecture of the convolutional neural network (CNN) used.

\paragraph*{Data Preprocessing.} We downloaded and preprocessed the training and test sets of MNIST \cite{mnist_dataset} dataset\footnote{We used the MNIST dataset from torchvision; see \href{https://docs.pytorch.org/vision/0.8/datasets.html\#mnist}{https://docs.pytorch.org/vision/0.8/datasets.html\#mnist} } and we normalized the images to be in $[-0.5, 0.5]^d$ where $d = 784$ is the number of pixels of the images as suggested in \cite{liu_svr}. 

\paragraph*{Network Architecture.} We then implemented the multiclass classifier with a convolutional neural network (CNN) consisting of five convolutional layers interleaved with ReLU activations and max pooling operations. The first two convolutional layers use $3\times3$ kernels with $1$ input channel and $32$ output channels, and then $32$ input channels and $64$ output channels, respectively. These are followed by a $2\times2$ max pooling layer with stride $2$. The third convolutional layer applies $3\times3$ filters to produce $64$ output channels, followed by another convolutional layer with $64$ input and output channels using $3\times3$ filters. A second $2\times2$ max pooling layer is applied thereafter. The resulting feature maps are flattened and passed through two fully connected layers with $200$ hidden units each, both activated by ReLU. The last layer is a fully connected layer with $10$ output units and softmax activation. Network architecture is summarized in Table \ref{tab:network_architecture}.

\begin{table}[h]
\centering
\caption{Network architecture}\label{tab:network_architecture}
\begin{tabular}{llll}
\toprule
Layer & Type & Parameters & Activation \\
\midrule
1 & Convolutional & in=$1$, out=$32$, kernel=$3\times3$ & ReLU \\
2 & Convolutional & in=$32$, out=$64$, kernel=$3\times 3$ & ReLU  \\
3 & Max Pooling & kernel=$2 \times 2$, stride=$2$ & --  \\
4 & Convolutional & in=$64$, out=$64$, kernel=$3 \times 3$ & ReLU  \\
5 & Convolutional & in=$64$, out=$64$, kernel=$3 \times 3$ & ReLU  \\
6 & Max Pooling & kernel=$2 \times 2$, stride=$2$ & --  \\
7 & Dense & in=$1024$, out=$200$ & ReLU  \\
8 & Dense & in=$200$, out=$200$ & ReLU  \\
9 & Dense & in=$200$, out=$10$ & Softmax  \\
\bottomrule
\end{tabular}
\end{table}

\paragraph*{Training Details.} Here, we describe the training procedure and the parameters used to train the classifier. In particular, we adopt the defensive distillation technique introduced in \cite{defensive_distillation} to improve robustness against adversarial perturbations. Defensive distillation is a two-stage training procedure based on a teacher–student framework. Let $\mathcal{D} = {(x_i, y_i)}_{i=1}^n$ denote the labeled training dataset, where $x_i \in \mathbb{R}^d$ and $y_i \in {1, \dots, C}$. In the first stage, a teacher network $f_\text{teacher}: \mathbb{R}^d \rightarrow \mathbb{R}^C$ is trained using the softmax function with temperature $T > 1$ i.e. $\text{softmax}_T(z)$ is a vector such that every entry $i$ is 
\begin{equation*}
[\text{softmax}_T(z)]_i = \frac{\exp(z_i / T)}{\sum_{j=1}^C \exp(z_j / T)},
\end{equation*}
where $z \in \mathbb{R}^C$ are the logits produced by the network. The teacher model is trained to minimize the standard cross-entropy loss between the softened predictions and the one-hot encoded true labels. In the second stage, we use the teacher network to produce soft labels i.e. let $\bar{y}_i = \text{softmax}_T(f_\text{teacher}(x_i))$ a student network $f_\text{student}$, with the same architecture as $f_\text{teacher}$, is trained to mimic the teacher’s soft predictions instead of hard labels $y_i$. 
After training, the temperature is reset to $T = 1$ for inference. In our experiments, we set the temperature to $T = 100$ for both teacher and student training stages, as suggested in \cite{defensive_distillation}. We partition the original training set into two disjoint subsets: $80$\% for training and $20\%$ for validation. All models are trained on the training part using the SGD optimizer with a learning rate of $0.01$, momentum $0.9$, batch size of $32$, and early stopping based on validation loss (we interrupt the training if for $5$ consecutive epochs (patience) the validation loss does not decrease). Moreover, we include dropout in the last layer before the output layer with a dropout rate of $0.8$. Training is conducted for up to $100$ epochs. Parameters are summarized in Table \ref{tab:training_params}, and in Table~\ref{tab:net_acc}, we report the classification accuracy obtained by the student model on the training and test sets.

\begin{table}[h]
    \centering
    \caption{Training parameters}\label{tab:training_params}
    \begin{tabular}{ll}
    \toprule
    Parameter &   \\
    \midrule
    Optimizer & SGD \\
    Learning rate & $0.01$ \\
    Momentum & $0.9$ \\
    Batch size & $32$ \\
    Max epochs & $100$ \\
    Temperature & $100$ \\
    Early stopping & Patience of $5$ epochs (on validation loss) \\
    Dropout & $0.8$  \\
    \bottomrule
    \end{tabular}

\end{table}

\begin{table}[H]
	\centering
	\caption{Training and test accuracies of the trained model}\label{tab:net_acc}
	\begin{tabular}{ll}
		\toprule
		 Classification Accuracy (\%) & \\
		 \midrule
		 Training & 98.8\\
		 Test & 98.5 \\
		 \bottomrule
		
	\end{tabular}
\end{table}

\paragraph*{Objective Function.} As described in Section \ref{sec:perf_metrics}, the target function is the black-box attacking loss proposed in \cite{ji_var_red,zoo,liu_svr}. Formally,  let $f_w : \mathbb{R}^d \rightarrow \mathbb{R}^C$ be a classifier parameterized by weights $w$, mapping an input image $z \in \mathbb{R}^d$ to a probability distribution over $C$ classes via the softmax output. Let $(z, y)$ an input-label pair from the MNIST test set correctly classified by the network $f_w$, where $z \in \mathbb{R}^d$ is the original image (normalized to be in $[-0.5, 0.5]^d$) and $y \in \{1, \ldots, C\}$ is the true label. Then, we consider the following objective function,
\begin{equation*}
    F(x) := \max\left\{ \log f_w(\psi(x, z) )_y - \max\limits_{j \neq y} \log f_w(\psi(x, z))_j, -\kappa \right\} + \frac{\lambda}{2} \|\psi(x, z) - z\|^2,
\end{equation*}
where $x \in \mathbb{R}^d$ is a perturbation, $f_w(\cdot)_j$ denotes the predicted probability for class $j$, $\lambda > 0$ is a regularization parameter and $\psi$ is the manipulation function defined as
\begin{equation*}
    \psi(x, z) = \frac{\tanh{(\tanh^{-1}(2z) + x)}}{2}.
\end{equation*}
In the experiments we set $\lambda = 50.0$ and $\kappa=1.0$. The initial condition $x_0$ is fixed for every method as a very large perturbation and it is $x_0 = [10.0,\cdots,10.0]^\top$. %

\paragraph*{Optimization parameter settings.} %
For the adversarial perturbation experiments, we run Algorithm \ref{algo:ffd} using the same parameter configuration for every method. Specifically, we set the initial step size to $\gamma_0 = 1.0$, and the Armijo condition constant to $c = 10^{-7}$. The minimum and maximum stepsize are set to $\gamma_{\text{min}} = 10^{-10}$ and $\gamma_{\text{max}}=1000$ respectively. The expansion and contraction factors for the line-search are set to $\theta = 0.9$ and $\rho = 1/\theta$, respectively. The finite-difference discretization parameter is held constant throughout the optimization process and is set to $h_k = 10^{-7}$ for all iterations.

\section{Limitations}\label{app:limitations}
In this appendix, we discuss the main limitations of this work. This study provides an empirical comparison of finite-difference algorithms using structured and unstructured direction sets. %
The primary objective of this work is to isolate, evaluate and study the effect of structure in direction generation on the performance of Algorithm \ref{algo:ffd}. The second goal is to introduce and analyze several efficient strategies for constructing structured direction matrices, highlighting their potential benefits in terms of gradient approximation quality and optimization performance. That said, a more extensive and systematic empirical evaluation would be necessary to fully characterize the strengths and limitations of various direction generation strategies across a broader range of optimization problems and practical scenarios. Conducting such a benchmark remains an important direction for future research. Additionally, our analysis focuses exclusively on the forward finite-difference estimator defined in eq. \eqref{eqn:forward_fd}. While this estimator is commonly used due to its simplicity and efficiency, alternative estimators (such as central finite differences or interpolation-based approaches \cite{Berahas2022}) are known to offer improved accuracy under certain conditions. Investigating the role of structured directions with these alternative estimators presents an interesting direction for future work. Furthermore, to isolate and study the effect of direction selection, we employed a line-search procedure to adaptively determine the step size, thereby avoiding manual tuning. However, it would be valuable to explore and compare different step-size strategies to understand how step-size choice interacts with direction selection across various settings.

\section{Further Experiments}\label{app:further_experiments}

To further support our findings, we report additional results on three benchmark functions: regularized logistic regression, Trid, and Griewank\footnote{The definitions of the Trid and Griewank functions are taken from \href{https://www.sfu.ca/~ssurjano/optimization.html}{https://www.sfu.ca/~ssurjano/optimization.html}}. In Table~\ref{tab:add_syn_fun_def}, we report the definitions of the additional synthetic functions used in the experiments. For each of these functions, we evaluate both the gradient approximation accuracy and convergence under various choices of direction matrices. In all experiments, we fix the discretization parameter to $h = 10^{-7}$. For the convergence experiments, we set a function evaluation budget of $10000$. The initial step size is set to $\gamma_0 = 1.0$, with minimum and maximum step sizes of $10^{-10}$ and $10^3$, respectively. The contraction and expansion factors are set to $0.5$ and $2.0$, and the Armijo condition parameter is fixed at $10^{-7}$. These additional results complement those presented in Section~\ref{sec:results} and further show the benefits of structured directions across a wider range of objective functions. In all experiments, the input dimension is fixed to $d = 500$. For the logistic regression task, we simulate a binary classification problem by generating $n = 1000$ data points ${(z_i, y_i)}_{i=1}^n$, where each $z_i \in \mathbb{R}^d$ is sampled from a standard multivariate Gaussian $\mathcal{N}(0, I)$, and labels are generated via $y_i = \text{sign}(\langle x^*, z_i \rangle)$, with $x^* \sim \mathcal{N}(0, I)$.

\begin{table}[h]
    \centering
    \caption{Definition of additional synthetic functions}
    \label{tab:add_syn_fun_def}
    \begin{tabular}{ll}
        \toprule
         Name & Definition \\
         \midrule
         Logistic & $F(x) =\frac{1}{n} \sum\limits_{i = 1}^n  \log(1 + e^{-\langle x, z_i \rangle y_i}) + \lambda \|x\|^2$ with $\lambda=10^{-5}$ \\
         Trid & $F(x) = \sum\limits_{i = 1}^{d} (x_i - 1)^2 - \sum\limits_{i = 2}^d x_i x_{i - 1}$\\
         Griewank & $F(x) = 1 + \sum\limits_{i = 1}^d \frac{x_i^2}{4000} - \prod\limits_{i=1}^d \cos\left(\frac{x_i}{\sqrt{i}}\right)$\\
         \bottomrule
    \end{tabular}
\end{table}

\noindent In Figure~\ref{fig:other_grad_acc}, we report the relative gradient approximation error achieved by the surrogate defined in eq. \eqref{eqn:forward_fd} using various strategies to construct the direction matrix. Consistent with the findings presented in Section \ref{sec:results}, we observe that structured methods generally yield lower approximation errors compared to unstructured ones, particularly when the number of directions $\ell \geq d/3$.

\noindent Figure~\ref{fig:other_convergence} shows the function value progress achieved by Algorithm~\ref{algo:ffd} when optimizing the synthetic functions defined in Table~\ref{tab:add_syn_fun_def}, using different approaches for generating direction matrices. The results further confirm the observations from Section~\ref{sec:results}: methods that rely on structured directions (excluding Householder) achieve comparable or superior convergence performance than the unstructured counterparts. As in Section~\ref{sec:results}, when $\ell = d$, all structured methods achieve similar performance with very low variance (i.e., the curves of the structured methods closely overlap with that of the QR method).

\begin{figure}[H]
    \centering
    \includegraphics[width=\linewidth]{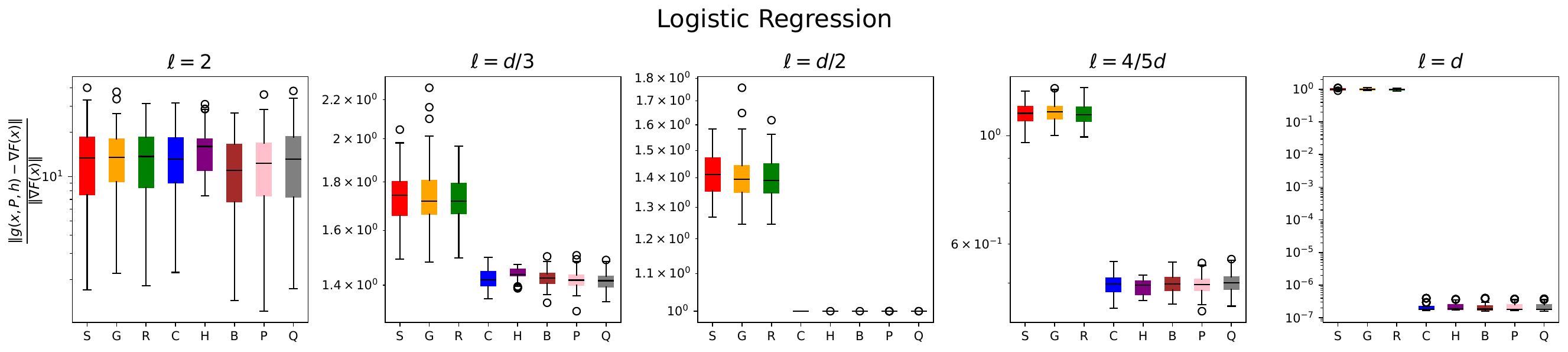}
    \includegraphics[width=\linewidth]{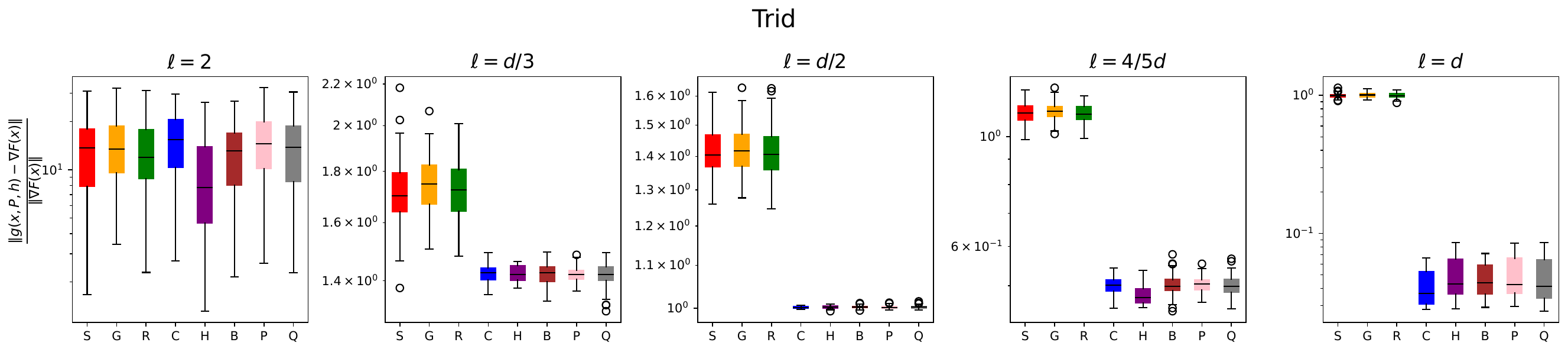}
    \includegraphics[width=\linewidth]{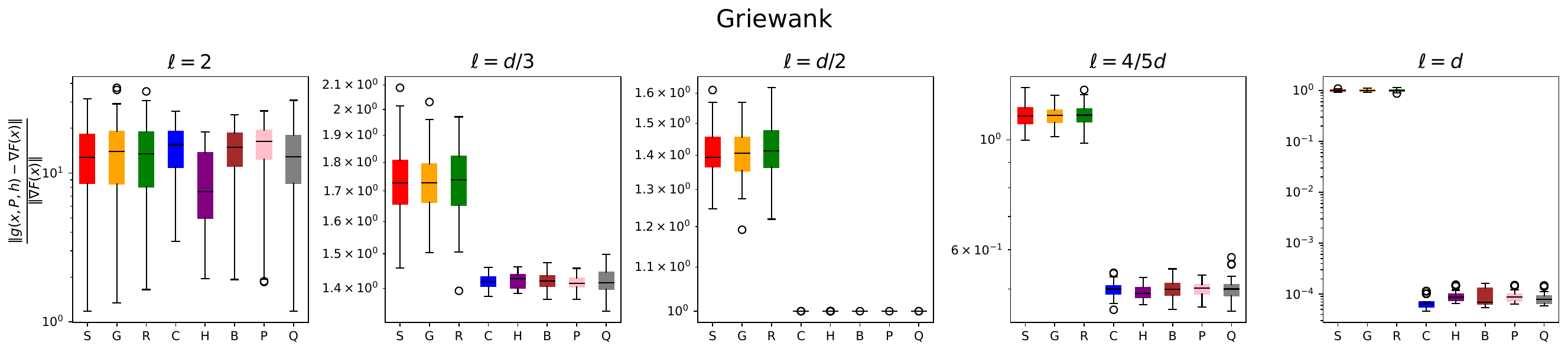}
    \caption{Relative gradient approximation error for the Logistic, Trid, and Griewank functions using the surrogate in Eq.~\eqref{eqn:forward_fd}, with direction matrices generated by: S (Spherical), G (Gaussian), R (Rademacher), C (Coordinate), H (Householder), B (Butterfly), P (Permuted Householder), Q (QR). Note that the scale of y-axis is different in the various plots.}
    \label{fig:other_grad_acc}
\end{figure}

\begin{figure}[H]
    \centering
    \includegraphics[width=\linewidth]{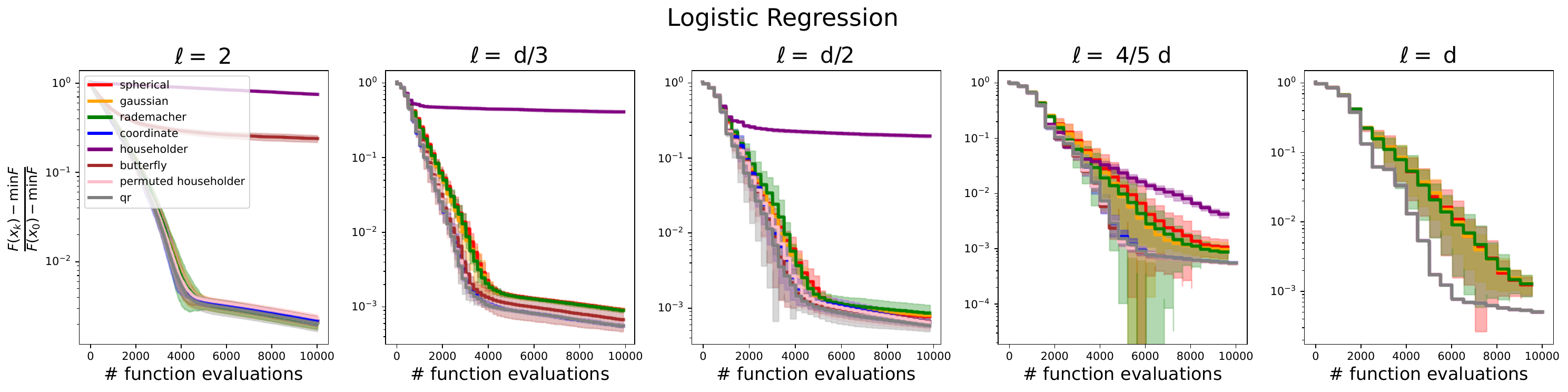}
    \includegraphics[width=\linewidth]{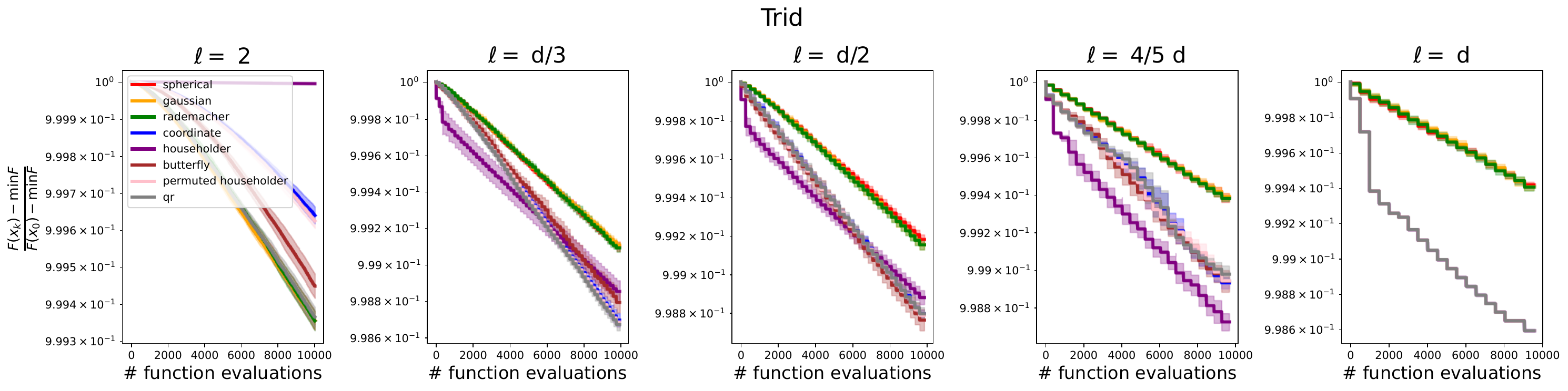}
    \includegraphics[width=\linewidth]{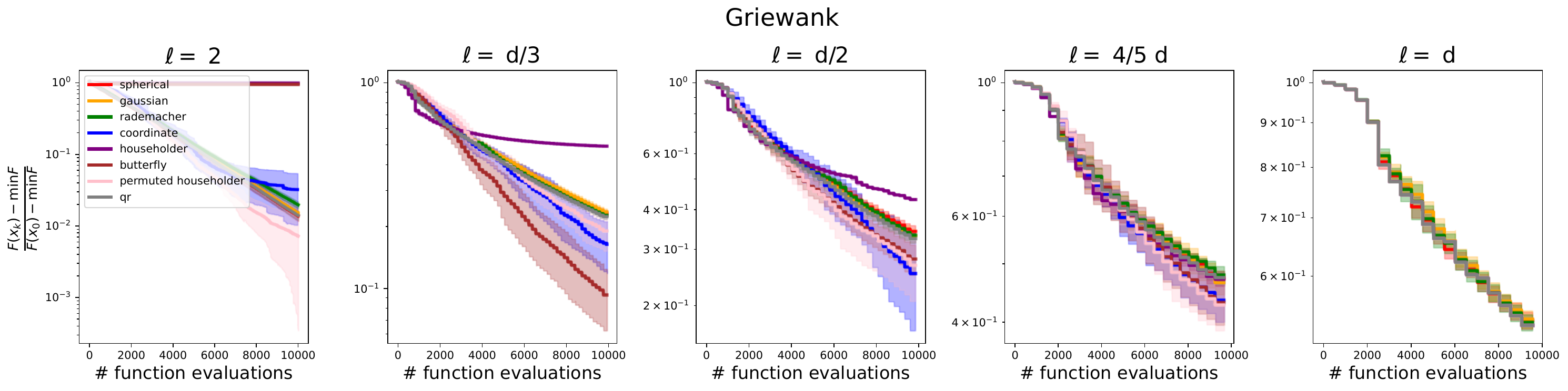}
    \caption{Function value progress in optimizing Logistic, Trid, and Griewank functions.}
    \label{fig:other_convergence}
\end{figure}

\section{Approximation error of structured and unstructured surrogates}\label{app:str_unstr}
In this appendix, we provide further details and the formal proof of the motivating example proposed in Section \ref{sec:preliminary_obs}. 
To this end, we introduce the concept of smoothing. Let $h > 0$ be a smoothing parameter and $\rho$ a probability measure on $\mathbb{R}^d$. The smooth surrogate of $F$ is
\begin{equation*}
    F_{h, \rho}(x) := \int F(x + hu) d\rho(u),
\end{equation*}
Several authors used this notion to analyze finite-difference algorithm for non-smooth setting - see e.g. \cite{rando2023optimal,duchi_power_of_two,nesterov2017random,flaxman2005online}. In the following, we consider $\rho$ to be the uniform measure over the unit ball $\mathbb{B}^{d} := \{ x\in\mathbb{R}^d \,|\,\|x\| \leq 1\}$ i.e. we define the following smooth surrogate
\begin{equation*}
    F_h(x) := \frac{1}{\text{vol}(\mathbb{B}^{d})} \int F(x + h u) du,
\end{equation*}
where $\text{vol}(\mathbb{B}^{d})$ denotes the volume of $\mathbb{B}^d$. Many works \cite{gasnikov_sph,flaxman2005online,Berahas2022}) show that the estimator in eq.~\eqref{eqn:forward_fd}, where the directions $(p^{(i)})_{i = 1}^\ell$ are sampled i.i.d. from the unit sphere, is an unbiased estimator of the gradient of such a surrogate $\nabla F_h(x)$ i.e. for every $x \in \mathbb{R}^d$, $\ell \geq 1$ and $h > 0$,
\begin{equation*}
    \nabla F_h(x) = \frac{d}{\ell} \sum\limits_{i = 1}^\ell \mathbb{E}_{p^{(i)} \sim\mathcal{U}(\mathbb{S}^{d - 1})}\left[\frac{F(x + h p^{(i)}) - F(x)}{h} p^{(i)}\right],
\end{equation*}
where $\mathcal{U}(\mathbb{S}^{d - 1})$ denotes the uniform distribution over the unit sphere $\mathbb{S}^{d-1} := \{v \in \mathbb{R}^{d} \, | \, \|v\| = 1\}$. More recently, it has been proved in \cite[Smoothing Lemma]{rando2023optimal} that it is also possible to obtain an unbiased estimator of $\nabla F_h(x)$ using structured directions. Specifically, if the directions are defined as $p^{(i)} = G e_i$, where $G$ is sampled uniformly from the orthogonal group $O(d) := \left\{ G \in \mathbb{R}^{d \times d} \,|\, \det G \neq 0 \, \wedge \, G^{-1} = G^\top \right\}$ with respect to the Haar measure, we have for every $x\in\mathbb{R}^d$, $\ell \geq 1$ and $h > 0$
\begin{equation*}
    \begin{aligned}
        \nabla F_h(x) &= \frac{d}{\ell} \sum\limits_{i = 1}^\ell \mathbb{E}_{p^{(i)} \sim\mathcal{U}(\mathbb{S}^{d - 1})}\left[\frac{F(x + h p^{(i)}) - F(x)}{h} p^{(i)}\right]\\
        &= \frac{d}{\ell} \sum\limits_{i = 1}^\ell \mathbb{E}_{G \sim \mathcal{U}(O(d)) } \left[ \frac{F(x + h Ge_i) - F(x)}{h} Ge_i \right],
    \end{aligned}
\end{equation*}
where $\mathcal{U}(O(d))$ denotes the uniform distribution over the orthogonal group. While both unstructured (spherical) and structured (orthogonal) directions yield unbiased estimators of the gradient of the smooth surrogate, we show with the following lemma that the estimator constructed using orthogonal directions achieves a smaller approximation error. 
\begin{lem}[Approximation Error]
    Let $P_1 = GI_{d,\ell}$ where $G$ is sampled uniformly from $O(d)$ with respect to the Haar measure and $P_2$ be a random matrix where every column is sampled i.i.d. from the unit sphere. Then, for every $x \in \mathbb{R}^d$ and $h  > 0$, the following inequality holds
    \begin{equation*}
        \mathbb{E}[\|g(x,h,P_1) - \nabla F(x)\|^2] \leq \mathbb{E}[\|g(x,h,P_2) - \nabla F(x)\|^2],
    \end{equation*}
    where $g$ is the finite difference approximation defined in eq. \eqref{eqn:forward_fd}.
\end{lem}
\begin{proof}
We start by bounding the approximation error for structured gradient approximation. Developing the square,
\begin{equation*}
    \|g(x,h,P_1) - \nabla F(x)\|^2 = \|g(x, h, P_1)\|^2 + \| \nabla F(x) \|^2 - 2\langle g(x, h, P_1),\nabla F(x) \rangle.
\end{equation*}
By eq. \eqref{eqn:forward_fd},
\begin{equation*}
    \begin{aligned}
    \|g(x,h,P_1) - \nabla F(x)\|^2 &= \frac{d^2}{\ell^2 h^2} \underbrace{ \left\| \sum\limits_{i = 1}^\ell (F(x + h G e_i) - F(x))G e_i \right\|^2 }_{(a)}\\
&+ \| \nabla F(x) \|^2 - 2\langle g(x, h, P_1),\nabla F(x) \rangle.\\
    \end{aligned}
\end{equation*}
Now, we focus on (a). Let $\delta_i^{G} = F(x + h G e_i) - F(x)$. Then, we have
\begin{equation*}
    \begin{aligned}
         \left\| \sum\limits_{i = 1}^\ell \delta_i^GG e_i \right\|^2 &= \left\langle \sum\limits_{i = 1}^\ell \delta_i^GG e_i, \sum\limits_{j = 1}^\ell \delta_j^GG e_j \right\rangle  = \sum\limits_{i=1}^\ell \sum\limits_{j=1}^\ell  \left\langle \delta_i^GG e_i,  \delta_j^GG e_j \right\rangle \\
        &=\sum\limits_{i=1}^\ell \| \delta_i^GG e_i \|^2 +\ \sum\limits_{j\neq i}  \left\langle \delta_i^GG e_i,  \delta_j^GG e_j \right\rangle .
    \end{aligned}
\end{equation*}
Since $\langle e_i, e_j \rangle = 0$ for every $i\neq j$, we have
\begin{equation*}
    \begin{aligned}
        \left\| \sum\limits_{i = 1}^\ell \delta_i^GG e_i \right\|^2 &= \sum\limits_{i=1}^\ell \| \delta_i^GG e_i \|^2.
    \end{aligned}
\end{equation*}
Thus, we have
\begin{equation*}
    \begin{aligned}
    \|g(x,h,P_1) - \nabla F(x)\|^2 &= \frac{d^2}{\ell^2 h^2} \sum\limits_{i=1}^\ell \| \delta_i^GG e_i \|^2+ \| \nabla F(x) \|^2 - 2\langle g(x, h, P_1),\nabla F(x) \rangle.\\
    \end{aligned}
\end{equation*}
Now, let $\delta_i^{p^{(i)}} = F(x + hp^{(i)}) - F(x)$. Repeating the same procedure for the approximation error using the finite-difference approximation with spherical directions, we get
\begin{equation*}
    \begin{aligned}
    \|g(x,h,P_2) - \nabla F(x)\|^2 &= \frac{d^2}{\ell^2 h^2} \underbrace{ \left\| \sum\limits_{i = 1}^\ell \delta_i^{p^{(i)}}p^{(i)} \right\|^2 }_{(b)}\\
&+ \| \nabla F(x) \|^2 - 2\langle g(x, h, P_2),\nabla F(x) \rangle,\\
    \end{aligned}
\end{equation*}
where the (b) term is
\begin{equation*}
    \begin{aligned}
         \left\| \sum\limits_{i = 1}^\ell \delta_i^{p^{(i)}} p^{(i)} \right\|^2  &=\sum\limits_{i=1}^\ell \| \delta_i^{p^{(i)}} p^{(i)} \|^2 +\ \sum\limits_{j\neq i} \left\langle \delta_i^{p^{(i)}} p^{(i)},  \delta_j^{p^{(j)}} p^{(j)} \right\rangle.
    \end{aligned}
\end{equation*}
Thus, summarizing, we have that using orthogonal directions
\begin{equation}\label{eqn:app_orth_dir}
    \begin{aligned}
        \|g(x,h,P_1) - \nabla F(x)\|^2 &= \frac{d^2}{\ell^2 h^2} \sum\limits_{i=1}^\ell \| \delta_i^GG e_i \|^2\\
        &+ \| \nabla F(x) \|^2 - 2\langle g(x, h, P_1),\nabla F(x) \rangle,        
    \end{aligned}
\end{equation}
while for independent spherical directions we have
\begin{equation}\label{eqn:app_unst_dir}
    \begin{aligned}
        \|g(x,h,P_2) - \nabla F(x)\|^2 &= \frac{d^2}{\ell^2 h^2}  \left( \sum\limits_{i=1}^\ell \| \delta_i^{p^{(i)}} p^{(i)} \|^2 +\ \sum\limits_{j\neq i} \left\langle \delta_i^{p^{(i)}} p^{(i)},  \delta_j^{p^{(j)}} p^{(j)} \right\rangle \right)\\
        &+ \| \nabla F(x) \|^2 - 2\langle g(x, h, P_2),\nabla F(x) \rangle.
    \end{aligned}
\end{equation}
Taking the expectation in eq. \eqref{eqn:app_orth_dir}, by \cite[Lemma 1]{rando2023optimal}, we get
\begin{equation*}
    \begin{aligned}
    \mathbb{E}[\|g(x,h,P_1) - \nabla F(x)\|^2] &= \frac{d^2}{\ell^2 h^2} \sum\limits_{i=1}^\ell \mathbb{E}[\| \delta_i^GG e_i \|^2] + \| \nabla F(x) \|^2 - 2\langle \nabla F_h(x),\nabla F(x) \rangle.   
    \end{aligned}
\end{equation*}
Moreover, by \cite[Theorem 3.7]{mattila_1995}, we have that $\mathbb{E}[\| \delta_i^GG e_i \|^2] = \mathbb{E}[\| \delta_i^{p^{(i)}} p^{(i)} \|^2]$ with $p^{(i)}$ sampled uniformly from the sphere $\mathbb{S}^{d-1}$. Thus, we have
\begin{equation}\label{eqn:app_orth_dir_2}
    \begin{aligned}
    \mathbb{E}[\|g(x,h,P_1) - \nabla F(x)\|^2] &= \frac{d^2}{\ell^2 h^2} \sum\limits_{i=1}^\ell \mathbb{E}[\| \delta_i^{p^{(i)}} p^{(i)} \|^2] + \| \nabla F(x) \|^2 - 2\langle \nabla F_h(x),\nabla F(x) \rangle.   
    \end{aligned}
\end{equation}
Taking the expectation in eq. \eqref{eqn:app_unst_dir}, by \cite[Lemma 1]{flaxman2005online}, we get
\begin{equation*}
    \begin{aligned}
        \mathbb{E}[\|g(x,h,P_2) - \nabla F(x)\|^2] &= \frac{d^2}{\ell^2 h^2}  \left( \sum\limits_{i=1}^\ell \mathbb{E} \left[\| \delta_i^{p^{(i)}} p^{(i)} \|^2\right] +\ \sum\limits_{j\neq i} \mathbb{E} \left[\left\langle \delta_i^{p^{(i)}} p^{(i)},  \delta_j^{p^{(j)}} p^{(j)} \right\rangle \right] \right)\\
        &+ \| \nabla F(x) \|^2 - 2\langle \nabla F_h(x),\nabla F(x) \rangle.
    \end{aligned}    
\end{equation*}
Since $p^{(i)}$, $p^{(j)}$ are independent, again by \cite[Lemma 1]{flaxman2005online}, we have
\begin{equation}\label{eqn:app_unst_dir_2}
    \begin{aligned}
        \mathbb{E}[\|g(x,h,P_2) - \nabla F(x)\|^2] &= \frac{d^2}{\ell^2 h^2}  \left( \sum\limits_{i=1}^\ell \mathbb{E} \left[\| \delta_i^{p^{(i)}} p^{(i)} \|^2\right] +\ \sum\limits_{j\neq i} \| \nabla F_h(x) \|^2\right)\\
        &+ \| \nabla F(x) \|^2 - 2\langle \nabla F_h(x),\nabla F(x) \rangle.
    \end{aligned}    
\end{equation}
Therefore equations \eqref{eqn:app_unst_dir_2} and \eqref{eqn:app_orth_dir_2} yield
\begin{equation}
 \mathbb{E}[\|g(x,h,P_1) - \nabla F(x)\|^2] = \mathbb{E}[\|g(x,h,P_2) - \nabla F(x)\|^2] -\frac{d^2}{\ell h^2}(\ell-1)\|\nabla F_h(x)\|,
\end{equation}
and the statement follows.

\end{proof}

\end{document}

%% file: MacroBook.tex
\def\argmin{\operatornamewithlimits{arg\,min}}

\newtheorem{lem}{Lemma}